\documentclass{amsart}
\usepackage{amssymb}
\usepackage{amsfonts}
\usepackage{stmaryrd}
\usepackage{mathrsfs}
\usepackage{graphicx}
\usepackage{color}

\begin{document}
\theoremstyle{thm}
\newtheorem{thm}{Theorem}[section]
\newtheorem{lem}[thm]{Lemma}
\newtheorem{cor}[thm]{Corollary}
\newtheorem{pro}[thm]{Proposition}
\newtheorem{cond}[thm]{Condition}
\newtheorem{exm}[thm]{Example}
\newtheorem{conj}[thm]{Conjecture}

\theoremstyle{definition}
\newtheorem{definition}[thm]{Definition}
\newtheorem{con}[thm]{Condition}
\newtheorem{clm}[thm]{Claim}
\newtheorem{fact}[thm]{Fact}
\newtheorem{rmk}[thm]{Remark}
\newtheorem{ques}[thm]{Question}
\def\square{\hfill${\vcenter{\vbox{\hrule height.4pt \hbox{\vrule
width.4pt height7pt \kern7pt \vrule width.4pt} \hrule height.4pt}}}$}
\def\T{\mathcal T}

\newenvironment{pf}{{\it Proof:}\quad}{\square \vskip 12pt}

\title{Width of a satellite knot and its companion}

\author{Qilong Guo, Zhenkun Li}

\thanks{The first author is supported by Science Foundation of China University of Petroleum,Beijing (No.2462015YJRC034 and No.2462015YQ0604). }


\begin{abstract}
In this paper, we give a proof of a conjecture  which says that $w(K) \geqslant n^2w(J)$, where $w(.)$ is the
width of a knot, $K$ is a satellite knot with $J$ as its companion, and $n$ is the winding number of the pattern. We also show that equality holds if $K$ is a satellite knot with braid pattern.
\end{abstract}

\maketitle

\vspace*{0.5cm} {\bf Keywords}: Width, Satellite knots, Companion, Pattern, Winding number.\vspace*{0.5cm}

AMS Classification (2010): 57M25, 57M27

\section{Introduction}
Width is an important invariant of knots which is introduced by Gabai in \cite{Ga}. It gives rise to the notion of thin position (of knots), which is essentially used in Gabai's proof of property R (see \cite{Ga}) and Gordon and Luecke's proof of the knot complement conjecture (see \cite{Go}), among others. We can view width as a kind of refinement of bridge number.
It is an interesting question how those knot invariants behave under the operations of connected sum and taking satellite. For bridge number, we know that $b(K_1\# K_2)=b(K_1)+b(K_2)-1$, and $b(K) \geqslant nb(J)$, given that $K$ is a satellite knot with companion $J$ and $n$ is the wrapping number (see \cite{Sc,Sch}).
In the case of width, it is conjectured that $w(K_1\# K_2)=w(K_1)+w(K_2)-2$ and $w(K) \geqslant n^2w(J)$,
which are both similar to bridge number.
However, the first conjecture is disproved by Blair and Tomova (see \cite{bt}).
For the second one, there is a weak version conjecturing that that $w(K) \geqslant n^2w(J)$ where $n$ is the winding number instead of the wrapping number.
Zupan (\cite{Zu1,Zu2}) proves that $w(K)\geqslant 8n^2$ where $n$ is the winding number and $w(K)=q^2w(J)$, where $K$ is a $(p,q)$-cable knot with companion $J$ and $q$ acts as the winding number.
Both of his results give partial positive answers to the weak version. In this paper, we give a complete positive answer to the weak version involving winding number.
\begin{thm}
Let $K$ be a satellite knot with companion $J$, and suppose the winding number of the pattern is $n$. Then
$$w(K) \geqslant n^2w(J).$$
\end{thm}

In section 2 we introduce some basic concepts and construct a graph associated to the neighborhood of the companion; in section 3 we prove that there is a simple loop in this graph, and such a loop is unique; in section 4 we associate each knot with a word in $Z_2$, the free monoid of rank 2, and then use it to help calculate the width.

\section{Preliminaries}
First we introduce some basic definitions.

\begin{definition}
Suppose $\hat{V}$ is a standard solid torus in $S^3$, and $\hat{k}$ is a knot in $int(\hat{V})$  such that $\hat{k}$ is not contained in any 3-ball $B\subset\hat{V}$.
Let $j \subset S^3$ be a non-trivial knot and let $V=N(j)$ be the closure of a tubular neighborhood of $j$ in $S^3$. Let $f:\hat{V} \rightarrow S^3$ be an embedding such that $f(\hat{V})=V$, and let $k=f(\hat{k})$. Then $k$ is called a $satellite~knot$ with $companion$ $j$ and $pattern$ $\hat{k}$. The $winding~number$ (of the pattern) is defined to be  the algebraic intersection number (up to a sign) of the pattern with a meridian disk. Furthermore, if $K$ (or $J,\hat{K}$) is the knot type represented by $k$ (or $j,\hat{k}$), we could say that $K$ is a satellite knot with companion $J$ and pattern $\hat{K}$ without ambiguity.
\end{definition}

Regard $S^3$ as the unit sphere in $\mathbb{R}^4$, and let $\pi:\mathbb{R}^4 \rightarrow \mathbb{R}$ be the projection $(x_1,x_2,x_3,x_4)\mapsto x_4$.
In the rest of this paper, we always assume that $h=\pi|_{S^3}$. Then $h$ is a Morse function on $S^3$ with exactly two critical points. These two critical points are $h^{-1}(1)$ and $ h^{-1}(-1)$, and we call them infinite (critical) points. For each $r\in (-1,1)$, $h^{-1}(r)$ is obviously a 2-sphere, which is called a level sphere.

\begin{definition}
Let $K$ be a knot type and let $\mathcal{K}$ be the set of all knots $k \in K$ such that
\begin{itemize}
  \item $k$ does not contain the two infinite points,
  \item $h|_{k}$ is Morse, and
  \item the critical points of $h|_{k}$ are in distinct levels.
\end{itemize}

For each $k \in \mathcal{K}$, suppose all the critical values of $h|_{k}$ are $c_1<c_2<...<c_m$. Choose regular values $r_1,r_2,...,r_{m-1}$ such that $c_i<r_i<c_{i+1}$ for $i=1,2,...,m-1$, and let $\omega_i(k)=|k \cap h^{-1}(r_i)|$. Define
\[w(k)=\sum_{i=1}^{m-1}{\omega_i(k)}\]
and
$$w(K)=\min_{k \in \mathcal{K}}{w(k)}.$$
$w(K)$ is called  the $width$  of the knot type $K$. See Figure 1 for the width of trefoil.
\end{definition}

\begin{figure}
\center
\includegraphics[width=0.7\textwidth]{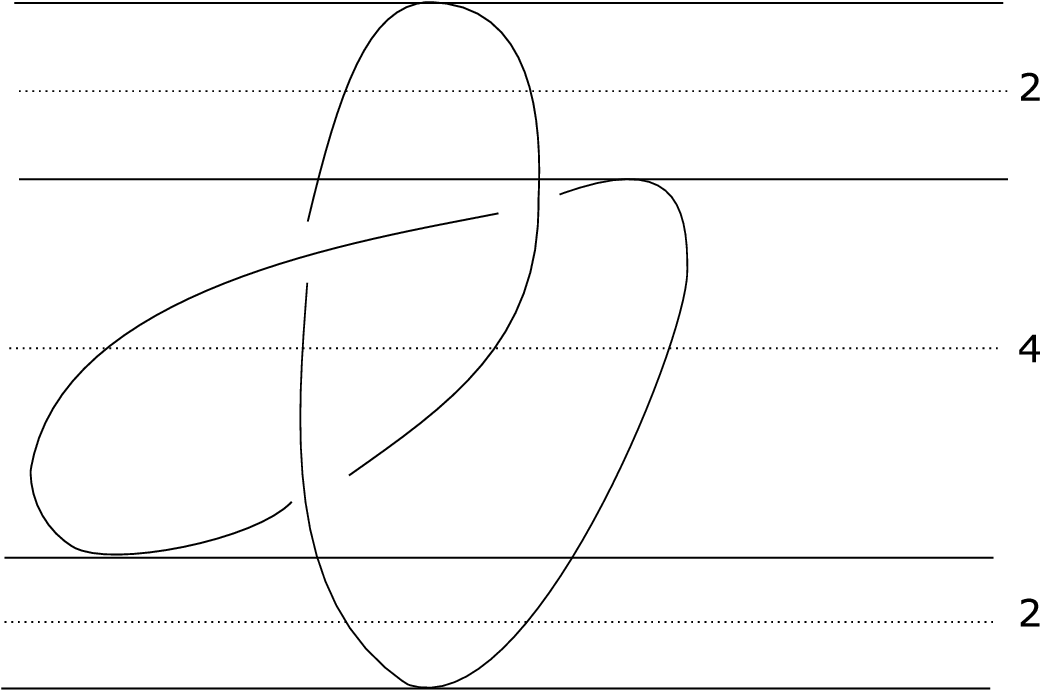}
\caption{The width of trefoil is 8}
\end{figure}


\vskip 0.3mm

In this paper, we focus on "nice" solid tori defined as follows.
\begin{definition}
Let $V$ be a solid torus in $S^3$. We say that $V$ is $nice$ if
\begin{itemize}
  \item $V$ does not contain the two infinite points,
  \item $h|_{\partial{V}}$ is also a Morse function,      and
  \item all critical points of $h|_{\partial{V}}$ are in distinct levels.
\end{itemize}
\end{definition}

If the original solid torus contains the infinities then the first condition can be achieved by 'digging them out', i.e. pick an arc connecting one infinity point and a point on the boundary $\partial V$ (disjoint from $k,j$), remove a tubular neighborhood of the arc and modify the new boundary to satisfy the other two conditions.


\begin{definition}
 Let $V$ be a nice solid torus in $S^3$. We construct a graph as follows. Let $c_1<c_2<...<c_m$ be all critical values of $h|_{\partial{V}}$, and let $M=V-( \cup_{i=1}^{m} h^{-1}(c_i))$.
Then vertices of the graph correspond to connected components of $M$ and edges correspond to components of $h^{-1}(c_i)\cap V$ for $i=1,2,...,m$, which are not points. we require that two vertices $v_1$ and $v_2$ are connected by an edge if and only if the two corresponding components of $M$ {are} separated by a component of $h^{-1}(c_i)\cap V$ which corresponds to the edge. We would like to call this graph the £¤$Reeb~graph$ and denote it by $\Gamma(V)$.
\end{definition}

We need some results from \cite{SS}. In general, critical points of $h|_{\partial{V}}$ are classified as maximal, minimal and saddle points. Following \cite{SS},  maximal (or minimal) points can be further divided into external and internal maximal (or minimal) points; saddle points can be divided into nested and unnested saddle points. We don't want to introduce the detailed definition here as they are not important for the use of this paper. Now the connectivity graph studied in \cite{SS} can be defined as follows.

\begin{definition}
Suppose $c'_1,...,c'_l$ are all critical values of $h|_{\partial{V}}$, corresponding to all external maximal, external minimal and unnested saddle points, and let $M'=V-( \cup_{i=1}^l h^{-1}(c'_i))$. If we carry out the construction in definition 2.4 using $M'$ and $c_i'$, then the graph we get is called the connectivity graph, and is denoted by $\Gamma_C{(V)}$.
\end{definition}

We can see from the definition that in order to get $M$ from $M'$, we need to cut off a second time those critical levels containing internal maximal, internal minimal and nested saddle points. Then a component of $M'$ is either unchanged or cut off into a few components. For graph this corresponds to that a vertex of $\Gamma_{C}(V)$ is either unchanged or replaced by some other graph which can be easily seen to be connected. Hence we have the following lemma:

\begin{lem}
Let $V$ be a nice solid torus, then $\Gamma(V)$ contains a loop if the connectivity graph $\Gamma_{C}(V)$ does.
\end{lem}

In this paper, we also need following results from \cite{SS}.

\begin{lem}
For any two knots $K_1,K_2$,
$w(K_1\sharp K_2)\geq max\{w(K_1),w(K_2)\}.$
\end{lem}

\begin{lem}
Let $V\subset S^3$ be a nice solid torus with boundary $T$.
Then there is an new embedding $i:V\rightarrow S^3$, satisfying following properties:

\begin{itemize}
\item $H=\overline{S^3-i(V)}$ is a solid torus.
\item The connectivity graph is a tree if and only if there is a meridian disk $D$ of $H$ such that $i^{-1}(\partial D)\subset T$ is horizontal in $T$ under $h|_T$
(i.e., $i^{-1}(\partial D)\subset (h|_T)^{-1}(r)$ for some $r$, where $r$ is a regular value of $h|_T$).
\end{itemize}

\end{lem}

{\bf Remark.} Lemma 2.7 is nothing but a special case of Proposition 2.3 in \cite{SS}.

We will also use the word 'vertex' to refer to its corresponding connected component of $M$. It is not hard to see that each vertex has a product structure $P \times (c_i,c_{i+1})$, where $P$ is a horizontal planar surface (or what we call a horizontal piece in definition 3.5) and $c_i,c_{i+1}$ are the two critical values that bound the vertex from below and above. We could embed $\Gamma(V)$ into $int(V)$ as follows.

\begin{enumerate}
  \item Pick one point in the interior of each component of $M$.

  \item For two adjacent vertices, connected the two points in the two vertices by a monotone decreasing arc in the interior of $V$.
\end{enumerate}

Then $\Gamma(V)$ can be thought of as a 1-dimensional complex in $\mathring{V}$, as in Figure 2. The solid curves are the boundary of the solid torus $V$ and the interior of $V$ is bounded by them. The dashed curves indicates the embedding of $\Gamma(V)$ into $\mathring{V}$.

\begin{figure}
\begin{center}
\includegraphics[width=0.7\textwidth]{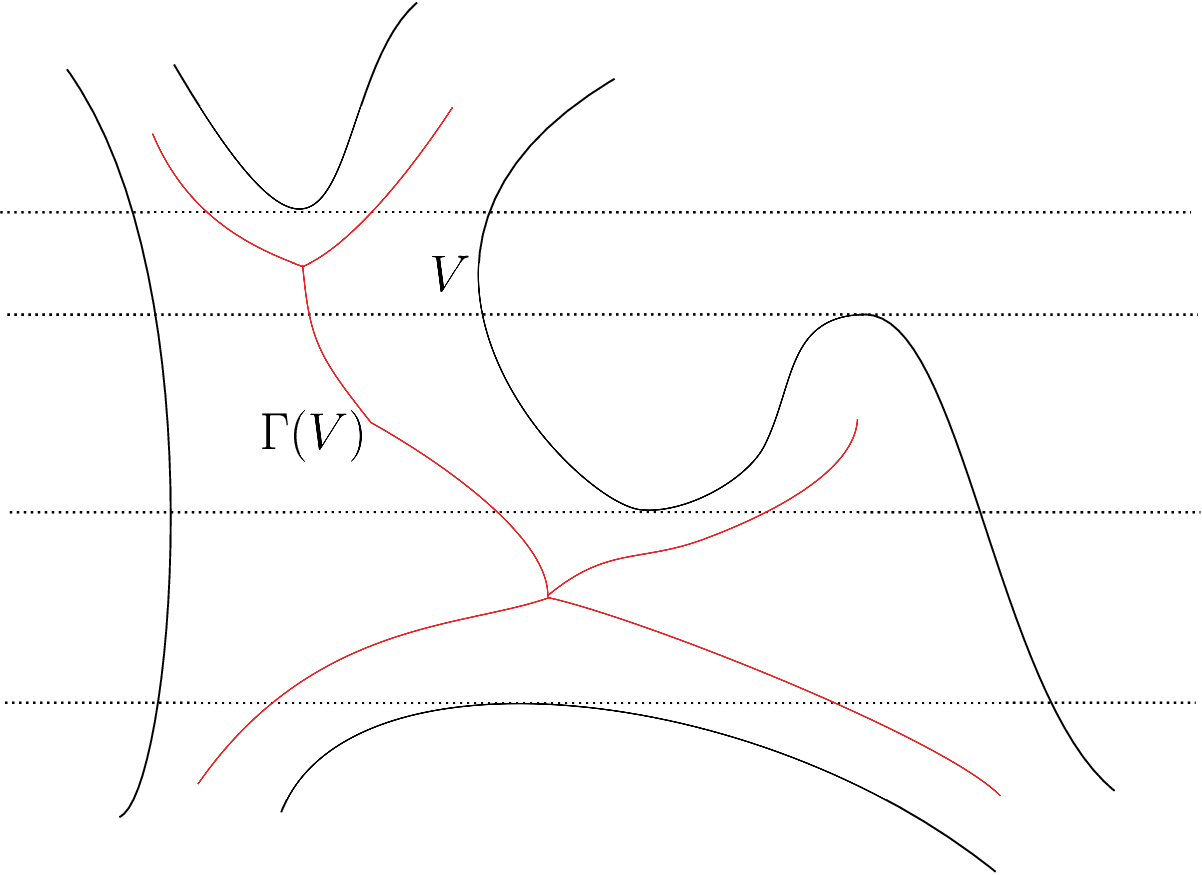}
\end{center}
\begin{center}
\caption{Embe $\Gamma(V)$ into $V$. }
\end{center}
\end{figure}

\section{$\Gamma(V)$ contains a unique simple loop}
In this section, we always assume that $V \subset S^3$ is a nice solid torus and $\Gamma(V)$ is the graph constructed as in Definition 2.4. We need some preliminary results before proving that there is a unique loop in $\Gamma(V)$.

\begin{lem}
Suppose $V$ is knotted in $S^3$, $T=\partial{V}$ and $r$ is a regular value of $h|_{T}$ such that one component of $(h|_{T})^{-1}(r)$ is an essential curve on $T$. Then at least one component $X$ of $h^{-1}(r) \cap V$ is a surface with boundary such that exactly one boundary component $\delta$ is  essential on $T$. Furthermore, $\delta$ is a meridian on $T$.
\end{lem}
\begin{proof}
Since $r$ is a regular value of $h|_{T}$, $(h|_{T})^{-1}(r)$ is {a} disjoint union of {some} simple closed curves. Note that $h^{-1}(r)$ is a sphere in $S^3$, so each component $\alpha$ of $(h|_{T})^{-1}(r)$, which is {essential in} $T$, bounds two disks on $h^{-1}(r)$. We say $\alpha$ is $innermost$ if $\alpha$ bounds a disk in $h^{-1}(r)$ which does not contain any other essential curve of $(h|_{T})^{-1}(r)$. By the innermost arguments, we can find a component $P$ of $h^{-1}(r)-T$, such that only one component of $\partial P$ is essential in $T$.

\vskip 0.3cm

{\bf Claim.}
$P \subset V$.

\vskip 0.3cm

{\it Proof of Claim.}
Suppose $P$ is not contained in $V$, then $P \cap \mathring{V}=\varnothing$. Let $\delta$ be the component of $\partial P$ which is essential in $T$, then {any} other component of $P$ bounds a disk in $T$, hence $\delta$ bounds a singular disk in $\overline{S^3-V}$. By Dehn's lemma, $\delta$ bounds an embedded disk in $\overline{S^3-V}$ and thus $T$ is compressible in $\overline{S^3-V}$, which contradicts to the assumption that $V$ is knotted.

By similar argument as above, we can see that $\delta$ bounds an embedded disk in $V$, hence $\delta$ is a meridian of $T$.
\end{proof}

\begin{cor}
Suppose $V$ is knotted and $r$ is a regular value of $h|_{T}$. Suppose $\delta$ is a component of $(h|_{T})^{-1}(r)$ which is essential in $T$, then $\delta$ is a meridian of $T$.
\end{cor}

\begin{cor}
The graph $\Gamma(V)$ is not a tree if $V$ is a knotted solid torus.
\end{cor}

\begin{proof}
Suppose, on the contrary, that $\Gamma(V)$ is a tree. By Lemma 2.5, the connectivity graph defined in \cite{SS} for $V$ is also a tree.
Let $i:V\rightarrow S^3$ be the embedding as in Lemma 2.7,
then there is a meridian disk $D$ of solid torus $H=\overline{S^3-i(V)}$ such that $i^{-1}(\partial D)\subset (h|_T)^{-1}(r)$ for some $r$. Obviously $i^{-1}(\partial D)$ is an essential curve in $T$, so by Corollary 3.2 $i^{-1}(\partial D)$ is a meridian of $V$, hence $\partial D$ bounds a disk in $i(V)$, which contradicts that $S^3=i(V)\cup H $.

\end{proof}

\begin{definition}
Let $l$ be a simple loop of $\Gamma$. A vertex of $l$ which is locally minimal(maximal) under $h$ is called a $minimal~(maximal)$ $vertex$.
We say that a vertex is a $critical~vertex$ if it is either minimal or maximal. A vertex which is neither minimal nor maximal is called a $vertical~vertex$.
\end{definition}

\begin{definition}
Let $r$ be a regular value of $h|_{T}$. We call each component of $h^{-1}(r) \cap V$ a $horizontal~piece$.
\end{definition}

Let $v$ be a vertex in $l$, regard $v$ as a component of $M$ (see the discussion at the end of section 2) and let $P$ be a horizontal piece in $v$. We want to describe the intersection of $l$ with $P$. If $v$ is a maximal vertex, then the two adjacent vertices in $l$ are both below $v$, and thus the part of $l$ in $v$ is an arc with one maximal point with respect to the height function $h$ and $P$ intersects $l$ in either 0,1 or 2 points. Note that the intersection at 1 point is not transversal because if we move the piece slightly above or below, then the intersection would be 0 or 2 points. A similar result holds for minimal vertices. If $v$ is vertical, then the part of $l$ in $v$ is a monotonic arc so every horizontal piece intersects $l$ exactly once.

Another observation is that the two vertices adjacent to a critical vertices must both be vertical so in any simple loop, vertical vertices must exist.

\begin{lem}
Let $l$ be a simple loop in $\Gamma(V)$. Then, as a simple closed curve in $V$, $l$ (with any orientation) represents a generator of $H_{1}(V)$.
\end{lem}

\begin{proof}
Let $v$ be a vertical vertex of $l$, and pick a horizontal piece of $v$. Then $P$ is a properly embedded surface in $V$ and intersects $l$ transversally once, {so the algebraic intersection number of $l$ and $P$ is just $\pm 1$, hence $l$ must be a generator of $H_1(V)$}.
\end{proof}

\begin{pro}
There is a unique simple loop in $\Gamma(V)$.
\end{pro}

\begin{proof}
The existence of a loop follows from Corollary 3.3. To show the uniqueness, suppose, on the contrary, there are two different simple loops $l_1$, $l_2$. If $l_1$ and $l_2$ do not have the same vertical vertices, then there is a vertical vertex $v$ with respect to one loop but not the other, say, with respect to $l_1$ but {not} $l_2$. {Then pick a generic horizontal piece in $v$ and calculate the intersection number of $l_2$ and $P$. Since the geometric intersection number is either $0$ or $2$ (see the discussion below Definition 3.5), the algebraic intersection number is never $\pm 1$ (actually is always 0). But the algebraic intersection number of $l_1$ and $P$ is $\pm 1$. So this contradicts to the fact that both $l_1$ and $l_2$ are generators of $H_1(V)$}.

Finally observe that if the two simple {loops} have same vertical vertices, then they must be the same loop. So we conclude that the simple loop must be unique.

\end{proof}

\section{The inequalities}
To calculate width, we use a technique coming from Zupan (see Section 5 in \cite{Zu2}).
Let $Z_2$ be the free monoid generated by $\{a,b\}$ and let $\phi:Z_2\rightarrow \mathbb{Z}$ be a homomorphism such that $\varphi(a)=2,\varphi(b)=-2$ and $\varphi({\alpha\beta})=\varphi(\alpha)+\varphi(\beta)$. For a word  $x=\alpha_1\alpha_2...\alpha_m\in Z_2$, where each $\alpha_j=a$ or $b$, write
$x_i=\alpha_1\alpha_2...\alpha_i\in Z_2$ and define
$$w(x)=\sum_{i=1}^{m}{\varphi(x_i)}.$$
For each knot $k \in \mathcal{K}$ (see Definition 2.2), associate a word $x=x(k)\in Z_2$ to it as follows: suppose all the critical points of $k$, from the lowest to highest, are $p_1,p_2,...,p_m$, then define $$x=x(k)=\alpha_1\alpha_2...\alpha_m,$$  where $\alpha_i=a$ if $p_i$ is a local minimal critical point and $\alpha_i=b$ if $p_i$ is a local maximal critical point. Let $w_i(k)$ be defined as in Definition 2.2. It is not hard to see that $w_i(k)=\varphi(x_i)$ and
$$w(k)=\sum_{i=1}^{m}w_i(k)=\sum_{i=1}^{m} \varphi{(x_i)}=w(x).$$

\begin{lem}
Suppose $x=\alpha_1 \alpha_2 ... \alpha_m \in F$ is a word.

(i) Suppose $\varphi(x_i) \geqslant 0$ for $i=1,2...,m$. Let $x'$ be a word obtained by deleting two letters $\alpha_i,\alpha_j$ in $F$, where $i<j$, $\alpha_i=a$ and $\alpha_j=b$, then $w(x) \geqslant w(x')$.

(ii) Suppose $x'$ is obtained from $x$ by exchanging two letters $\alpha_i,\alpha_{i+1}$ where $\alpha_i=a$ or $\alpha_{i+1}=b$, then $w(x) \geqslant w(x')$.
\end{lem}

The proof is straightforward. We call the operation on words in (i)(or in (ii)) of above lemma the type I (or II) operation. The next lemma is useful when estimating $w(k)$:

\begin{lem}
Suppose $n$ is a fixed positive integer and $x=x(\tilde{k})$ is a word associated with a knot $\tilde{k}$ has the form $\omega_1\alpha_1^{s_1}\omega_2\alpha_2^{s_2}...\omega_{m}\alpha_m^{s_m}\omega_{m+1}$, where $\omega_i=\beta_{i1}\beta_{i_2}...\beta_{it_i}$ is a word for $i=1,2,...,m+1$.
Assume that each $s_i\geq n$.

Furthermore, suppose that  $x=\alpha_1...\alpha_m$ is the word associated to another knot $\hat{l}$ and

\begin{enumerate}
\item $\varphi(\omega_1\alpha_1^{s_1}...\omega_i\alpha_i^{s_i})\geqslant n\varphi(\alpha_1...\alpha_i),~\text{if}~\alpha_i=a$,
\item $\varphi(\omega_1\alpha_1^{s_1}...\omega_{i-1}\alpha_{i-1}^{s_{i-1}}\omega_i)\geqslant n\varphi(\alpha_1...\alpha_{i-1}),~\text{if}~\alpha_i=b$,
\end{enumerate}

Then we have
$$w(\tilde{k})\geqslant n^2w(\alpha_1\alpha_2...\alpha_m)=n^2w(\hat{l}).$$
\end{lem}

\begin{proof}
Suppose $1\leqslant i\leqslant m$. If $\alpha_i=a$, we have for $0\leqslant j\leqslant n-1$,
$$\varphi(\omega_1\alpha_1^{s_1}...\omega_i\alpha_i^{s_i-j})\geqslant n\varphi(\alpha_1...\alpha_i)-2j;$$
if $\alpha_i=b$, we have for $1\leqslant j\leqslant n$,
\begin{equation*}
\begin{aligned}
\varphi(\omega_1\alpha_1^{s_1}...\omega_{i-1}\alpha_{i-1}^{s_{i-1}}\omega_i\alpha_i^{j})&\geqslant n\varphi(\alpha_1...\alpha_{i-1})-2j\\
&=n(\varphi(\alpha_1...\alpha_{i})+2)-2j\\
&=n\varphi(\alpha_1...\alpha_{i})+2n-2j.
\end{aligned}
\end{equation*}
Since the word comes from a knot, $\varphi(x_l)\geqslant 0$ for any $l$. Hence we have:
\begin{equation*}
\begin{aligned}
w(\tilde{k})&\geqslant \sum_{\alpha_i=a}\sum_{j=0}^{n-1}{\varphi(\omega_1\alpha_1^{s_1}...\omega_i\alpha_i^{s_i-j})}+\sum_{\alpha_i=b}\sum_{j=1}^{n}{\varphi(\omega_1\alpha_1^{s_1}...\omega_{i-1}\alpha_{i-1}^{s_{i-1}}\omega_i\alpha_i^{j})}\\
&\geqslant \sum_{i=1}^{m}{n^2\varphi(\alpha_1,...,\alpha_i)}+\sum_{\alpha_i=a}\sum_{j=0}^{n-1}{(-2j)}+\sum_{\alpha_i=b}\sum_{j=1}^{n}{2n-2j}\\
&=n^2w(\alpha_1,...,\alpha_n)-\sum_{\alpha_i=a}{n(n-1)}+\sum_{\alpha_i=b}{n(n-1)}.\\
&=n^2w(\alpha_1,...,\alpha_n)\\
&=n^2w(\hat{l}).
\end{aligned}
\end{equation*}
\end{proof}

Now suppose $k$ is a satellite knot with companion $j$, and let $V$ be a closed regular neighborhood of $j$ that contains $k$. Without loss of generality, we can assume that $V$ is a nice solid torus. Let $l$ be the unique loop in $\Gamma(V)$ as in Proposition 3.7, then $l$ can also be viewed as a knot in $\mathring{V}\subset S^3$.
Denote $J$ (or $K,L$) the knot type of $j$ (or $k,l$).

\begin{lem}
$w(L) \geqslant w(J).$
\end{lem}

\begin{proof}
{Picking a horizontal piece $P$ as in Lemma 3.1, and capping off all inessential boundaries of $P$ near the boundary $T$, we get a meridian disk $D$ of $V$ such that $D\cap l=P\cap l$. By Lemma 3.6, $l$ represents a generator of $H_1(V)$, so the algebraic intersection number of $l$ and $P$ is $\pm1$. Note also that $l$ intersects $P$ (and hence $D$) at most 2 points (see the discussion below Definition 3.5). So $l$ must intersect the meridian disk $D$ (transversally) only once and hence can be viewed as a composition of $j$ and possibly another knot $l'$. Then by Lemma 2.6, $w(L)\geqslant w(J)$.}
\end{proof}

\begin{lem}
Let $P$ be a horizontal piece in a vertical vertex of $l$. Then the geometric intersection number of $k$ and $P$ is no less than the winding number of the pattern $k$.
\end{lem}

\begin{proof}

Since the winding number for the pattern of $k$ is $n$, and $l$ represents a generator of $H_1(V)\cong \mathbb{Z}$, we have $[k]=\pm n[l]\in H_1(V)$. Since $P$ is a vertical vertex, algebraic intersection number (up to sign) of $\pm1$ and $l$ is 1. Consequently, $P$ and $k$ must have algebraic intersection number (up to sign) $\pm n$, hence the geometric intersection number is at least $n$.
\end{proof}

\vspace{0.3cm}

Now we will isotope $l$ into an equivalent knot $\widehat{l}$ and change $k$ into another knot $\tilde{k}$ as follows.

Suppose all the critical points of $l$, are $q_1,q_2,...,q_m$, from the lowest to the highest. Each $q_j$ corresponds to a critical vertex of $\Gamma(V)$ and hence corresponds to a component of $M=V-( \cup_{i=1}^{n} h^{-1}(c_i))$ (see Definition 2.4), denoted by $C_j$.
When $q_j$ is a local minimal point of $l$, suppose $C_j$ is bounded from above by  $h^{-1}(c_{i_j})$. Since $v_j$ has a product structure, we can move $q_j$ up to a point $\widehat{q_j}$ so that $\widehat{q_j}$ is a critical point of $h|_T$ and $h(\widehat{q_j})=c_{i_j}$. Furthermore, we can assume that no more critical points of $l$ are created. Do similar operations on local maximal points of $l$ and after all such operations, $l$ becomes a new knot $\hat{l}$. Obviously $\hat{l}$ and $l$ are equivalent knots and all the critical points of $\widehat{l}$, are $\widehat{q_1},\widehat{q_2},...,\widehat{q_m}$, from the lowest to the highest. See Figure 3 for the isotopy near a local minimal point of $l$.

\begin{figure}
\center
\includegraphics[width=0.7\textwidth]{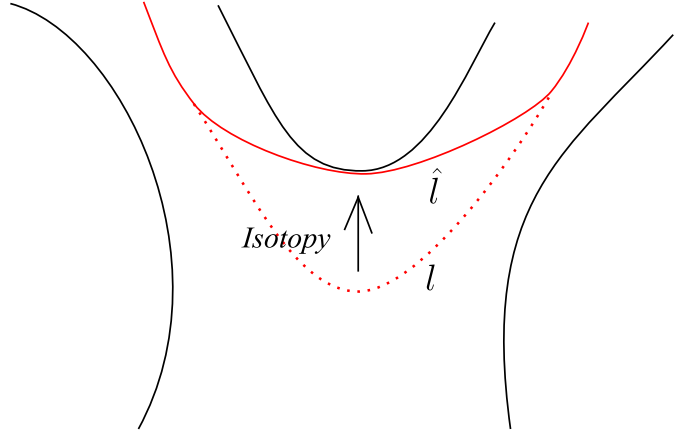}
\caption{Isotope $l$}
\end{figure}

Suppose $\widehat{q_j}$ is a local minimal point, then pick a regular value $r_{j}$ slightly larger than $c_{i_j}$ so that no other critical point of $T,\hat{l},k$ lies between two level spheres $h^{-1}(r_j)$ and $h^{-1}(c_{i_j})$. There are two horizontal pieces $P_j$ and $Q_j$ on $h^{-1}(r_{j})$ which respectively belong to the two vertical vertices adjacent to the vertex $C_j$. $P_j$ and $Q_j$ cut $k$ into arcs, and each arc of $k-P_j-Q_j$ that lies in $C_{j}$ is disjoint from any other $C_{t}$ for $t \neq j$, since $l$ is the unique simple loop. If $\beta$ is such an arc intersecting $C_j$, we can create a new arc $\gamma$ so that $\beta$ and $\gamma$ have the same end points, $\gamma$ has exactly one inner critical point which is local minimal and $\gamma$ is contained in $h^{-1}(c_{i_j},r_j)$. Then replace $\beta$ by $\gamma$ and do this repeatedly until no arc of $k-P_j-Q_j$ intersects $C_{j}$. Then we finish the operation for a particular local minimal point of $\hat{l}$.  See Figure 4. Do similar replacements for local maximal points of $\hat{l}$. After such replacement for all $m$ critical points of $\hat{l}$, $k$ becomes a new knot $\tilde{k}$. $k$ and $\tilde{k}$ may have different knot type, but it does not matter. We only need the following inequality.
\begin{figure}
\center
\includegraphics[width=0.7\textwidth]{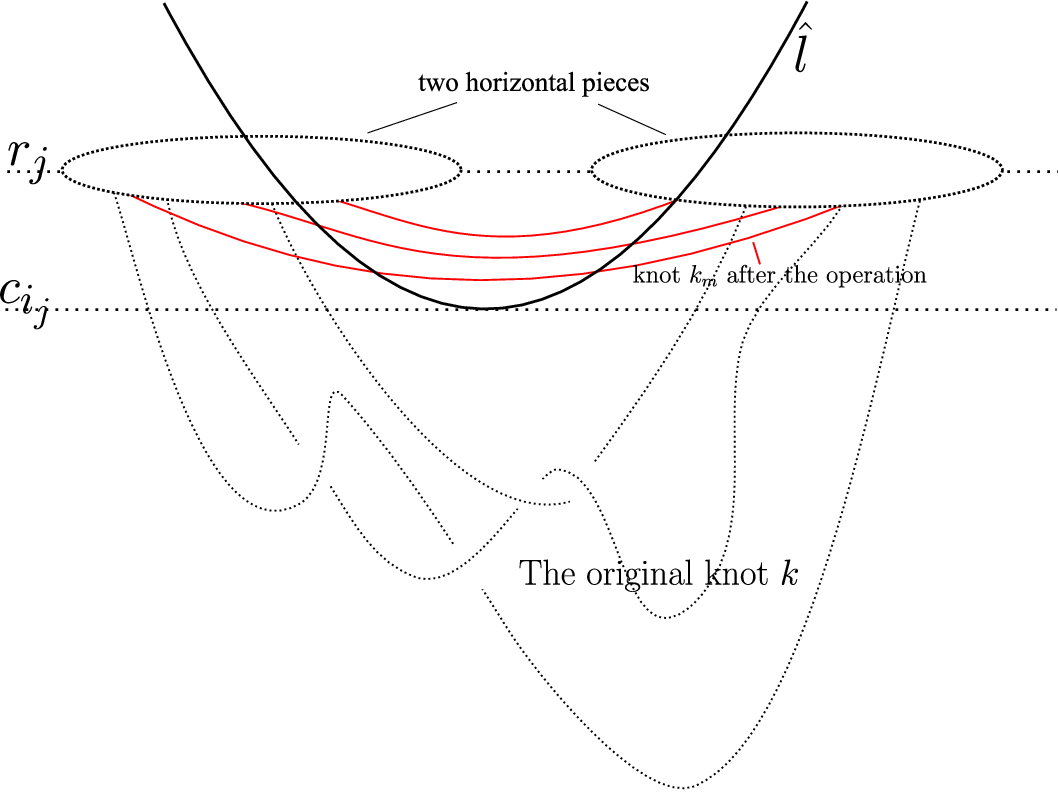}
\caption{Operation on $k$}
\end{figure}

\vspace{0.3cm}

\begin{lem}
$w(k)\geqslant w(\tilde{k})$.
\end{lem}

\begin{proof}
We study how $k$ becomes $\tilde{k}$. Let $\widehat{q_j}$ be a local minimal points of $\hat{l}$, let $P_j$, $Q_j$ be as above. Also let $\beta$ be an arc in $K-P_j-Q_j$ which has end points in $P_i\cup Q_i$ and has interior below them. Then the operation of creating $\gamma$ and replacing $\beta$ can be done by two step. the first step is to cancel pairs of maximal and minimal  points of $\beta$, to make $\beta$ has a unique critical point which is minimal. Since the interior of $\beta$ is below its two end points, we can always pair a maximal point with another minimal point which is lower. This corresponds to type I operation on words and by Lemma 4.1 will not increase width. The condition that $\phi(y_i)>0$ in Lemma 4.1 holds all the time because after cancelling each pair of points, $k$ still remains a knot. The second step is to lift the unique minimal point of $\beta$ above the level $c_{i_j}$. This corresponds to type II operation on words and will not increase width. Similar arguments apply to local maximal points of $\hat{l}$.
\end{proof}

\begin{lem}
$w(\tilde{k})\geqslant n^2w(\hat{l})$, where $n$ is the winding number.
\end{lem}

\begin{proof}
Now we need to estimate $w(\tilde{k})$. The difficulty is that we do not know every critical point of $\tilde{k}$ but only the ones near a critical point of $\hat{l}$. For a local minimal point $\widehat{q_j}$ of $\hat{l}$, let $r_i$ be a regular value slightly larger than $c_{i_j}$ as in the discussion above. Then $|h^{-1}(r_j) \cap \hat{l}|=\omega_j(\hat{l})$. Then there are exactly $\omega_j(\hat{l})$ horizontal pieces of $V$ on $h^{-1}(r_j)$, which intersect $\hat{l}$. By lemma 4.4, we have
$$|h^{-1}(r_j) \cap \tilde{k}| \geqslant n \omega_j(\hat{l}).$$
A similar argument applies when $\widehat{q_i}$ is local maximal point. The difference is that we should pick a regular level $r_i$ slightly lower than $c_{i_j}$ and hence
$$|h^{-1}(r_j) \cap \tilde{k}| \geqslant n (\omega_{j-1}(\hat{l})).$$
Note it is $\omega_{j-1}$ on the right because we pick a regular level lightly lower than the critical level $c_{i_j}$. Compare definition 2.1.

Suppose the word for $\hat{l}$ is $\alpha_1\alpha_2...\alpha_m$ then the word for $\tilde{k}$ can be written as $\omega_1\alpha_1^{s_1}\omega_2\alpha_2^{s_2}...\omega_{m}\alpha_m^{s_m}\omega_{m+1}$, where $\omega_i$ is an arbitrary word for $i=1,2,...,m+1$ and $s_i\geqslant n$ for all $i=1,2,...,m$. The argument above shows that the words for $\tilde{k}$ and $\hat{l}$ satisfies the conditions of Lemma 4.2 and hence we have $$w(\tilde{k})\geqslant n^2w(\alpha_1\alpha_2...\alpha_m)=n^2w(\hat{l}).$$
\end{proof}

\begin{thm}
Let $K$ be a satellite knot with companion $J$ and winding number $n$, Then $$w(K) \geqslant n^2w(J).$$
\end{thm}
\begin{proof}

Choose $k\in\mathcal{K}$ which realizes the width of its knot type. Without loss of generality, we can assume that $V$ is a closed regular neighborhood of $j$ which is nice and contains $k$. Furthermore we could assume all critical points of $\partial V$ and $k$ are in distinct levels.
Construct $\Gamma(V)$ and pick the unique loop $l$ by Proposition 3.7. Combining Lemma 4.3, 4.5 and 4.6, we have:
$$w(K)=w(k) \geqslant n^2w(\tilde{k}) \geqslant n^2w(\hat{l})\geqslant w(L)\geqslant w(J).$$
\end{proof}

\begin{cor}
Let $K$ be a satellite knot with knotted companion $J$ and the winding number of the pattern is $n$. If $K$ has a braid pattern,
then :
$$w(K)=n^2w(J).$$
\end{cor}
\begin{proof}
By theorem 4.5, we only need to show that $w(K) \leqslant n^2w(J)$. Pick an embedding $j$ so that $j$ realize the width of its knot type. Suppose the word associated to $j$ is $x=\alpha_1\alpha_2...\alpha_t$. Embed $k$ so that its critical points are all very near some critical point of $j$ and is as less as possible. Since $k$ has a braid pattern with winding number $n$, $k$ can be embedded so that the word associated with $k$ is exactly $y=\alpha_1^n\alpha_2^n...\alpha_m^n$. {By direct calculation}, we have:
$$w(K)\leqslant w(k) =w(y)=n^2w(x)=n^2w(j)=n^2w(J).$$
\end{proof}

 \vspace{3mm}

\begin{flushleft}
Qilong Guo,\\
College of Science,
China University of Petroleum-Beijing,\\
Beijing, China, 102249;\\
$Email$: guoqilong1984@hotmail.com

Zhenkun Li, \\
Mathematics Department,
Massachusetts Institute of Technology, \\
Cambridge, MA, 02139;\\

$Email$: zhenkun@mit.edu
\end{flushleft}

\end{document}